\documentclass[reqno,12pt,a4paper]{amsart}

\usepackage[utf8]{inputenc}

\usepackage{enumerate}
\usepackage{amstext,amsmath,amsthm,amsfonts,amssymb,amscd}
\usepackage{latexsym,mathrsfs,dsfont,euscript}
\usepackage{mathbbol}
\usepackage{mathtools}
\usepackage{xcolor}

\usepackage{fullpage}

\usepackage{hyperref} 
\hypersetup{
	colorlinks,%
	citecolor=blue,%
	filecolor=red,%
	linkcolor=red,%
	urlcolor=blue
}

\newtheorem{theorem}{Theorem}[section]
\newtheorem{proposition}[theorem]{Proposition}
\newtheorem{lemma}[theorem]{Lemma}

\theoremstyle{definition}

\theoremstyle{remark}

\numberwithin{equation}{section}

\newcommand{\abs}[1]{\left\vert#1\right\vert}

\newcommand{\norm}[1]{\left\Vert#1\right\Vert}

\allowdisplaybreaks

\begin{document}
	
%%%%%%%%%%%%% Metadata %%%%%%%%%%%%%%%%%%%%%%%%%%%%%%%%%%%
%% Title
\title[Multilinear approximate identities generated by hypermetrics on spaces of homogeneous type]{Multilinear approximate identities generated by hypermetrics on spaces of homogeneous type}

%%    Information for authors

\author[]{Hugo Aimar}
%\email{haimar@santafe-conicet.gov.ar}
%%%
\author[]{Ivana G\'{o}mez}
%\email{ivanagomez@santafe-conicet.gov.ar}
%%%
\author[]{Joaqu\'in Toledo}
%\email{joaquintoledo@santafe-conicet.gov.ar}
%%%
	
%%  Palabras claves y codigos


%
%
%%%%%%%%%%%%%%%%%%%%%%%%%%%%%%%%%%%%%%%%%
\begin{abstract}
The classical Newtonian potentials, defined in terms of metrics, give rise to the basic family of kernels defining linear integral operators and posing the fundamental problems of linear harmonic analysis. When the binary character of a metric on a set is naturally generalized to the $(k+1)$-ary character of hypermetric on the set, we obtain families of kernels of $k+1$ variables leading to multilinear integral operators of order $k$ or $k$-linear operators. In this paper we consider the problem of multilinear approximation to the multilinear identity through potentials built on hypermetrics in the general setting of spaces of homogeneous type.
\end{abstract}
%%%%%%%%%%%%%%%%%%%%%%%%%%%%%%%%%%%%%%%%%%

\maketitle

%%%%%%%%%%%%%%%%%%%%%%%%%%%%%%%%%%%%%%%%%%%%%
\section{Introduction}\label{sec:introduction}
In this paper we shall deal with special multilinear operators defined by kernels built on higher order metrics or hypermetrics. A simple observation that can be taken as starting point to illustrate the basic definitions is the following. The usual distance in the real numbers $\mathbb{R}$, given by $|x-y|$ can be seen as a constant times the distance in $\mathbb{R}^2$ of the point $(x,y)$ to the diagonal of $\mathbb{R}^2$. In fact given $x\neq y$ in $\mathbb{R}$, the orthogonal projection of $(x,y)$ on the diagonal $\{(z,z)\in \mathbb{R}^2: z\in \mathbb{R}\}=\bigtriangleup_2$ is the point $\frac{1}{2}(x+y,x+y)$. Hence if $\rho(x,y)$ denotes the distance of the point $(x,y)$ to $\bigtriangleup_2$, we have that $\rho^2(x,y)=\left(x-\frac{(x+y)}{2}\right)^2+\left(y-\frac{(x+y)}{2}\right)^2=\frac{1}{2}|x-y|^2$. In other words $|x-y|=\sqrt{2}\rho(x,y)$. If we have $m\geq 2$ real numbers $\{x_1,\ldots, x_m\}$ we can define in $\mathbb{R}^m$ a function $\rho(x_1,\ldots,x_m)$ as the distance of the point $(x_1,\ldots, x_m)$ to the diagonal $\bigtriangleup_m=\{(z,\ldots,z): z\in \mathbb{R}\}$ in $\mathbb{R}^m$. This procedure can be carried over the setting of metric or quasi-metric spaces. In this paper we aim to explore some of the basic multilinear operators generated by these hypermetrics on spaces of homogeneous type.

We shall consider here the problem of approximating the $k$-linear identity $I(f_1,\ldots,\allowbreak f_k)(x)=\prod_{i=1}^k f_i(x)$ (see \cite{GrafakosBook2014}), through $k$-linear integral operators whose kernels are given in terms of $\rho(x,x_1,\ldots,x_k)$ on a space of homogeneous type $(X,d,\mu)$, with $m=k+1$.  
In the special case of $\alpha$-Ahlfors regular spaces of homogeneous type, we give sufficient conditions on $\varphi$ in order to prove that
\begin{equation}\label{eq:multilinearoperatorphimeans}
	\frac{1}{J(x,\varepsilon)}\int\limits_{X^k}\varphi\left(\frac{\rho(x,x_1,\ldots,x_k)}{\varepsilon}\right) \prod_{i=1}^k f_i(x_i) d\mu(x_1)\ldots d\mu(x_k)
\end{equation}
tends to $\prod_{i=1}^k f_i(x)$ for almost every $x\in X$, when $f_i\in L^{p_i}(X,\mu)$ and $1=\sum_{i=1}^k \frac{1}{p_i}$, $1<p_i\leq\infty$, where $J(x,\varepsilon)=\int_{X^k}\varphi\bigl(\frac{\rho(x,x_1,\ldots,x_k)}{\varepsilon}\bigr) d\mu(x_1)\ldots d\mu(x_k)$ (see Theorem~\ref{thm:ApproximationIdentitySpaceLebesgue} for the proof of the convergence of \eqref{eq:multilinearoperatorphimeans} to $\prod_{i=1}^k f_i(x)$). When $k=1$, we have the classical linear case, see for example \cite{SteinBook1970}. The problem of boundedness of multilinear maximal functions has been considered for different maximal type operators including the Calder\'{o}n bilinear maximal. See for example \cite{Lacey2000}, references therein and \cite{DeTaThi2008}.
	
The paper is organized as follows. Section~\ref{sec:hypermetrics} is devoted to the definition and basic properties of hypermetrics in spaces of homogeneous type. In Section~\ref{sec:MultilinearMaximalEstimates} we prove that, under suitable conditions on $\varphi$, the maximal operator associated to \eqref{eq:multilinearoperatorphimeans} is pointwise bounded by a multilinear Hardy-Littlewood type maximal operator on the sections of $\rho$-tubes about the diagonal of $X^{k+1}$. We also prove in Section~\ref{sec:MultilinearMaximalEstimates} the boundedness properties of this Hardy-Littlewood type operator in the general setting of spaces of homogeneous type. Section~\ref{sec:TheApproximationtotheMultilinearIdentity} deals with the convergence of the means \eqref{eq:multilinearoperatorphimeans} to the multilinear identity when $\varepsilon$ tends to zero.

%%%%%%%%%%%%%%%%%%%%%%%%%%%%%%%%%%%%%%%%%%%%%%%%%%%%%%%%%%%%%%

\section{Hypermetrics}\label{sec:hypermetrics}

Let us recall that a quasi-metric space $(X,d)$ is a set $X$ with a function $d: X\times X \to \mathbb{R}_{\geq 0}$ that is symmetric, vanishing on the diagonal of $X\times X$, and only on the diagonal, satisfying the generalized triangle inequality $d(x,z)\leq \kappa(d(x,y)+d(y,z))$ for some constant $\kappa\geq 1$ and every $x, y$ and $z$ in $X$. When $\kappa=1$, $(X, d)$ is a metric space. Actually as it was proved in \cite{MaSe79Lipschitz} and \cite{Gus74}, every quasi-metric space is metrizable and moreover, the quasi-metric $d$ is equivalent to a power of a metric. We shall write $B_d(x,r)$ to denote the ball centered at $x$ with radius $r>0$, i.e. $B_d(x,r)=\{y\in X: d(x,y)<r\}$.

Given a quasi-metric space $(X,d)$ and a positive integer $k$ the space $(X^{k+1}, d^{k+1})$, where $X^{k+1}=X\times \cdots \times X$, $k+1$ times, and $d^{k+1}(\boldsymbol{x},\boldsymbol{y})=\sup\{d(x_i,y_i): i=0,1,\ldots,k\}$, with $\boldsymbol{x}=(x_0,x_1,\ldots,x_k)$ and $\boldsymbol{y}=(y_0,y_1,\ldots,y_k)$, is also a quasi-metric space with the same triangle constant $\kappa$ as $(X,d)$. Set $\bigtriangleup_{k+1}$ to denote the diagonal of $X^{k+1}$. Precisely $\bigtriangleup_{k+1}=\{\boldsymbol{x} \in X^{k+1}: x_0=x_1=\ldots=x_k\}$. As usual the distance of a point $\boldsymbol{x}$ to a set $E$ is defined as the infimum of the distances of the point $\boldsymbol{x}$ to the points $\boldsymbol{y} \in E$. We shall say that the function $\rho : X^{k+1} \to \mathbb{R}_{\geq 0}$ given by $\rho(\boldsymbol{x})=\rho(x_0,\ldots,x_k)=d^{k+1}(\boldsymbol{x},\bigtriangleup_{k+1})=\inf_{\boldsymbol{y}\in \bigtriangleup_{k+1}}d^{k+1}(\boldsymbol{x},\boldsymbol{y})$ is the \textbf{hypermetric of order $k+1$ on $X$ induced by $d$}.

In our analysis of multilinear operators defined by $\rho$ we shall be specially interested in the neighborhoods of the diagonal $\bigtriangleup_{k+1}$ of the type $V(r)=\{\boldsymbol{x}\in X^{k+1}: \rho(\boldsymbol{x})<r\}$. In particular we shall deal with its sections obtained by fixing one of the components of $\boldsymbol{x}$. Assume for example that we fix $x\in X$ and consider the subset of $X^{k}$ given by $E(x,r)=\{(x_1,\ldots,x_k)\in X^{k}:\rho(x,x_1,\ldots, x_k)<r\}$ which is the section at $x$ of $V(r)$. Notice that since $\rho(x_0,x_1,\ldots, x_k)=\rho(x_{\sigma(0)},x_{\sigma(1)},\ldots, x_{\sigma(k)})$ for every permutation $\sigma$ 
of $\{0,1,\ldots,k\}$, we have that $E(x,r)=E(x_i,r)$ for every choice of $i=1,\ldots,k$. The balls in $(X,d)$ have a quantitative control of the sets $E(x,r)$ that is given in the next result. 

\begin{lemma}\label{lemma:sectionesentrebolas}
Let $\kappa$ be the triangular constant for $d$, then 
\begin{align*}
\prod_{i=1}^{k}	B_{d}(x,r)\subset E(x,r)\subset \prod_{i=1}^{k} B_{d}(x,2\kappa r)
\end{align*}
for every $x\in X$ and every $r>0$.
\end{lemma}

\begin{proof}
	Take $(x_1,\ldots, x_k)\in \prod_{i=1}^{k}	B_{d}(x,r)$, then  $d(x,x_i)<r$, $i=1,\ldots,k$. And, of course $d(x,x)=0<r$. So that $\rho(x,x_1,\ldots, x_k)\leq d^{k+1}\left((x,x_1,\ldots, x_k), (x,x,\ldots, x)\right)=\sup_{i=1,\ldots,k}d(x,x_i)<r$, and $(x_1, x_2,\ldots,x_k)\in E(x,r)$. Let us now check the second inclusion. For $(x_1,\ldots, x_k)\in E(x,r)$ we have $\rho(x,x_1,\ldots, x_k)<r$, hence for some $u \in X$, $d^{k+1}\left((x,x_1,\ldots, x_k),(u,u,\ldots,u)\right)<r$. From the definition of $d^{k+1}$ we get $d(x,u)<r, d(x_1,u)<r, \ldots, d(x_k,u)<r$. Now, from the triangle inequality, $d(x_i,x)\leq \kappa(d(x_i,u)+ d(u,x))<2\kappa r$, for $i=1,\ldots,k$. These inequalities prove that $(x_1, x_2, \ldots, x_k)$ belongs to $\prod_{i=1}^{k} B_{d}(x,2\kappa r)$.
\end{proof}

As usual, see \cite{CoiWe1971,MaSe79Lipschitz} we shall say that $(X,d,\mu)$ is a space of homogeneous type if $(X,d)$ is a quasi-metric space and $(X,\mu)$ is a positive measure space on a $\sigma$-algebra containing the $d$-balls, such that there exists a constant $A$ for which $0<\mu(B_d(x,2r))\leq A\mu(B_d(x,r))<\infty$ for every $x\in X$ and every $r>0$. We say that $\kappa$ and $A$ are the geometric constants of $(X,d,\mu)$.

In a space of homogeneous type the above lemma provides estimates for the measures of the sections $E(x,r)$ of $V(r)$. The results in \cite{MaSe79Lipschitz} show that, if necessary, the quasi-metric can be changed by another quasi-metric that is continuous. Hence the $d$-balls are open in $X$, the tubes $V(r)$ are open sets in $X^{k+1}$ and their sections $E(x,r)$ are open sets in $X^k$. 

The next result follows immediately from Lemma~\ref{lemma:sectionesentrebolas}.
\begin{lemma}
	If $(X,d,\mu)$ is a space of homogeneous type such that the $d$-balls are open sets and $k$ is a positive integer, then $(X^{k},d^{k},\mu^{k})$, with $\mu^k$ the product measure on $X^k$, is a space of homogeneous type with geometric constants $\kappa$ and $A^k$. Moreover, the family $E(x,r)$ satisfies a doubling property $0<\mu^{k}\left(E(x,2r)\right)\leq \widetilde{A}\mu^{k}\left(E(x,r)\right)<\infty$, with $\widetilde{A}=(8\kappa)^{k\log_2A}$.
\end{lemma}	
\begin{proof}
The first claim follows readily from the fact that $(X,d,\mu)$ is a space of homogeneous type and from the definition of $d^k$. In order to check the second claim, notice first that the positivity and finiteness of the measure $\mu^k$ of the sections $E(x,r)$ follow from the same properties for the $d$-balls and Lemma~\ref{lemma:sectionesentrebolas}. Actually, the doubling property for the sections follows also from the quantitative consequences of Lemma~\ref{lemma:sectionesentrebolas}. In fact, since for $j\in\mathbb{N}$ such that $2^j\leq 4\kappa<2^{j+1}$ we have $\mu(B_d(x,4\kappa r))\leq A^{j+1}\mu(B_d(x,r))\leq A A^{\log_2 4\kappa}\mu(B_d(x,r))=2^{\log_2A}2^{\log_2 A \cdot\log_2 4\kappa}\mu(B_d(x,r))=(8\kappa)^{\log_2 A}\mu(B_d(x,r))$, then
\begin{equation*}
\mu^k(E(x,2r))\leq\prod_{i=1}^k\mu(B_d(x,4\kappa r)) \leq (8\kappa)^{k\log_2 A}\mu^k\left(\prod_{i=1}^k B_d(x,r)\right)\leq (8\kappa)^{k\log_2 A}\mu^k\left(E(x,r)\right).
\end{equation*}
\end{proof}
When the space of homogeneous type is $\alpha$-Ahlfors regular (see \cite{David1984}), for some positive $\alpha$, we have a uniform behavior for the measure of the sections $E(x,r)$. Recall that $(X,d,\mu)$ is $\alpha$-Ahlfors regular if there exist constants $0<\gamma\leq \Gamma<\infty$ such that
	\begin{align*}
		\gamma r^\alpha\leq \mu \left(B_d(x,r)\right)\leq \Gamma r^\alpha
	\end{align*}
for every $x\in X$ and every $r>0$. Notice that every $\alpha$-Ahlfors regular space is a space of homogeneous type with $A=2^\alpha\tfrac{\Gamma}{\gamma}$. From Lemma~\ref{lemma:sectionesentrebolas} we immediately have the next result.		
\begin{lemma}\label{lemma:measureSectionsEonAhlfors}
	If $(X,d,\mu)$ is $\alpha$-Ahlfors regular, then $\mu^k\left(E(x,r)\right)\simeq r^{k\alpha}$
	for every $x\in X$ and every $r>0$. Precisely, $c_1r^{k\alpha}\leq\mu^k(E(x,r))\leq c_2 r^{k\alpha}$ with $c_1=\gamma^k$ and $c_2=(2\kappa)^{\alpha k}\Gamma^k$.
\end{lemma}

%%%%%%%%%%%%%%%%%%%%%%%%%%%%%%%%%%%%%%%%
	
\section{Multilinear maximal estimates}\label{sec:MultilinearMaximalEstimates}
In this section $(X,d,\mu)$ is an $\alpha$-Ahlfors regular space which is complete in the metric sense. Some of the results will be proved under more general hypotheses that we shall be explicitly point out.

Given $k$-functions $f_i:X\to\mathbb{R}$, the $k$-linear identity operator is defined on the vector function $(f_1,f_2,\ldots,f_k)$ by
\begin{equation*}
	I(f_1,f_2,\ldots,f_k)(x)=\prod_{i=1}^k f_i(x).
\end{equation*}
From H\"{o}lder inequality we have that $\norm{I(f_1,f_2,\ldots,f_k)}_{L^1(X,\mu)}\leq \prod_{i=1}^k\norm{f_i}_{L^{p_i}(X,\mu)}$ whenever $1=\sum_{i=1}^k\tfrac{1}{p_i}$. In the search of approximations to $I$, through $k$-linear integral operators concentrating their singularities about the diagonal $\triangle_{k+1}$ ok $X^{k+1}$, we consider kernels of the form
\begin{equation*}
	\Phi_r(x,x_1,\ldots,x_k)=\frac{1}{J(x,r)}\varphi\left(\frac{\rho(x,x_1,\ldots,x_k)}{r}\right),
\end{equation*}
for $r>0$, $x\in X$ and $(x_1,\ldots,x_k)\in X^k$, where
\begin{equation*}
	J(x,r)=\int_{X^k}\varphi\left(\frac{\rho(x,x_1,\ldots,x_k)}{r}\right) d^k\mu(x_1,\ldots,x_k).
\end{equation*}
The next result provides sufficient conditions on $\varphi$ in order to get an estimate for $J(x,r)$ in terms of $r>0$ uniformly in $x\in X$.
\begin{lemma}\label{lemma:JcomparableRadio}
	Let $\varphi:\mathbb{R}_{\geq 0}\to \mathbb{R}_{\geq 0}$ be a nonincreasing function such that 
	\begin{equation*}
		0<s(\varphi)=\int_0^\infty \varphi(t)t^{k\alpha-1}dt<\infty.
	\end{equation*} 
	Then there exist two constants $C_1$ and $C_2$ depending only on $k$, $\gamma$, $\Gamma$ and $\kappa$ such that
	\begin{equation*}
		C_1s(\varphi)\leq\frac{J(x,r)}{r^{k\alpha}}\leq C_2 s(\varphi)
	\end{equation*}
	for every $r>0$ and every $x\in X$.
\end{lemma}
\begin{proof}
Take $x\in X$ and $r>0$, then for $\lambda>1$
\begin{align*}
	J(x,r)=\sum_{j\in \mathbb{Z}}\int_{r \lambda^{j}\leq \rho(x,x_1,\ldots,x_k)<r \lambda^{j+1}}\varphi\left(\frac{\rho(x,x_1,\ldots,x_k)}{r}\right)d^k\mu(x_1,\ldots,x_k). 
\end{align*}
Hence from the nonincreasing condition on $\varphi$ and Lemma~\ref{lemma:measureSectionsEonAhlfors}
\begin{align}
	J(x,r) &\leq \sum_{j\in \mathbb{Z}}\varphi\left(\lambda^{j}\right)\int_{r \lambda^{j}\leq \rho(x,x_1,\ldots,x_k)<r \lambda^{j+1}}d^k\mu(x_1,\ldots,x_k) \notag\\ \notag
	&\leq \sum_{j\in \mathbb{Z}}\varphi\left(\lambda^{j}\right) \mu^{k}\left(E(x,r \lambda^{j+1})\right)\\ \notag
	&\leq \Gamma^{k}(2\kappa)^{\alpha k}r^{\alpha k}\sum_{j\in \mathbb{Z}}\varphi\left(\lambda^{j}\right)\lambda^{k\alpha(j+1)}\\ \notag
	&=\frac{\lambda^{k\alpha}\Gamma^k (2\kappa)^{\alpha k}r^{\alpha k}}{\log \lambda}\sum_{j\in \mathbb{Z}}\int_{\lambda^{j-1}}^{\lambda^{j}}\varphi\left(\lambda^{j}\right) \lambda^{k\alpha j}\frac{dt}{t}\\ \notag
	&\leq \frac{\lambda^{2k\alpha}\Gamma^k (2\kappa)^{\alpha k}r^{\alpha k}}{\log \lambda}\sum_{j\in \mathbb{Z}}\int_{\lambda^{j-1}}^{\lambda^{j}}\varphi\left(t\right) t^{k\alpha-1}dt\\ 
	&=r^{k\alpha}\left(\frac{\lambda^{2k\alpha}}{\log \lambda}\right)\Gamma^k (2\kappa)^{k\alpha} s(\varphi).\label{eq:upperboundJ}
\end{align}
Let us consider now the lower bound for $J(x,r)$.		
From Lemma~\ref{lemma:sectionesentrebolas} we have that for $\lambda>2\kappa$,
\begin{equation*}
	E(x,r\lambda^{j+1})\setminus E(x,r\lambda^{j})\supset \prod_{i=1}^k B_d(x,r\lambda^{j+1})\setminus\prod_{i=1}^k B_d(x,2\kappa r\lambda^{j}).
\end{equation*}
Now, 
\begin{align*}
	J(x,r)&\geq \sum_{j\in \mathbb{Z}}\varphi\left(\lambda^{j+1}\right)\mu^k\left(E(x,r\lambda^{j+1})\setminus E(x,r\lambda^{j}) \right)\\
	& \geq \sum_{j\in \mathbb{Z}}\varphi\left(\lambda^{j+1}\right)\mu^k\left(\prod_{i=1}^k B_d(x,r\lambda^{j+1})\setminus\prod_{i=1}^k B_d(x,2\kappa r\lambda^{j})\right).
\end{align*}
From the Ahlfors character of the space,
\begin{align*}
\mu^k &\left(\prod_{i=1}^k B_d(x,r\lambda^{j+1})\setminus\prod_{i=1}^k B_d(x,2\kappa r\lambda^{j})\right)\\
&\phantom{\prod_{i=1}^k B_d(x,r\lambda^{j+1})\setminus} =\mu^k\left(\prod_{i=1}^k B_d(x,r\lambda^{j+1})\right)- \mu^k\left(\prod_{i=1}^k B_d(x,2\kappa r\lambda^{j})\right)\\
&\phantom{\prod_{i=1}^k B_d(x,r\lambda^{j+1})\setminus} \geq \gamma^k (r\lambda^{j+1})^{k\alpha}-\Gamma^k(2\kappa r\lambda^{j})^{k\alpha}\\
&\phantom{\prod_{i=1}^k B_d(x,r\lambda^{j+1})\setminus} =r^{k\alpha}\lambda^{k\alpha j}\left[\gamma^k\lambda^{k\alpha}-\Gamma^k 2^{k\alpha}\kappa^{k\alpha}\right].
\end{align*}
Hence, if we choose $\lambda_0=\dfrac{(1+\Gamma^{k}(2\kappa)^{k\alpha})^{\tfrac{1}{k\alpha}}}{\gamma^{\tfrac{1}{\alpha}}}$, we get
\begin{align*}
	J(x,r)&\geq r^{k\alpha}\sum_{j\in \mathbb{Z}}\varphi(\lambda_0^{j+1})\lambda_0^{k\alpha j}\\
	&\geq \frac{r^{k\alpha}}{(\lambda_0^{k\alpha})^2}\frac{1}{\log \lambda_0}\sum_{j\in \mathbb{Z}}\int_{\lambda_0^{j+1}}^{\lambda_0^{j+2}}\varphi(t)t^{k\alpha} \frac{dt}{t}\\
	&\geq \frac{r^{k\alpha}}{\lambda_0^{2k\alpha}\log \lambda_0} s(\varphi). 
\end{align*}
This inequality together with \eqref{eq:upperboundJ}, with $\lambda_0$ instead of $\lambda$, gives the result.
\end{proof}

Let us now introduce the $k$-sublinear maximal operator defined by the family $\{\Phi_\varepsilon: \varepsilon>0\}$ of kernels 
\begin{equation}\label{eq:MaximalOperatorfamilyPhi}
	\Phi^*(f_1,\ldots,f_k)(x)=\sup_{\varepsilon>0}\int_{X^k}\Phi_\varepsilon(x,x_1,\ldots,x_k)\prod_{i=1}^k |f_i(x_i)|d\mu^k(x_1,\ldots,x_k).
\end{equation}
It is not difficult to show, applying the generalized Schur Lemma in \cite{GraTo01}, that for fixed $\varepsilon>0$ the operator $\int_{X^k}\Phi_\varepsilon(x,x_1,\ldots,x_k)\prod_{i=1}^k |f_i(x_i)|d\mu^k(x_1,\ldots,x_k)$ is bounded from $L^{p_1}\times\ldots\times L^{p_k}$ into $L^q$ when $\sum_{i=1}^k\tfrac{1}{p_i}=\tfrac{1}{q}$ with $1<p_1,\ldots,p_k,q<\infty$.
The basic fact in the search of pointwise convergence of $\int_{X^k}\Phi_\varepsilon(x,x_1,\ldots,x_k)\prod_{i=1}^k |f_i(x_i)|d\mu^k$ is the pointwise boundedness of $\Phi^*(f_1,\ldots,f_k)(x)$ by a Hardy-Littlewood multilinear operator built on the sections $E(x,r)$ of the neighborhoods of the diagonal $\triangle_{k+1}$. Precisely, the adequate Hardy-Littlewood type maximal operator is defined by
\begin{equation}
	\mathcal{M}(f_1,\ldots,f_k)(x)=\sup_{r>0}\frac{1}{\mu^k(E(x,r))}\int_{E(x,r)}\prod_{i=1}^k |f_i(x_i)|d\mu^k(x_1,\ldots,x_k)
\end{equation}
for measurable functions $f_1,\ldots,f_k$ on $X$. 
Let us start by stating and proving an upper pointwise estimate of $\Phi^*$ by $\mathcal{M}$.
\begin{theorem}\label{thm:PointwiseEstimateforPhistarbyMaximalHLsections}
Let $(X,d,\mu)$ be an $\alpha$-Ahlfors regular space. Let $\varphi: \mathbb{R}_{\geq 0} \to \mathbb{R}_{\geq 0}$ be a nonincreasing function satisfying $0<s(\varphi)=\int_{0}^{\infty}\varphi(t)t^{k\alpha-1}dt<\infty$. Then there exists a constant $C$ depending only on $\kappa$, $\gamma$, $\Gamma$, $\alpha$ and $k$ such that
\begin{equation*}
	\Phi^*(f_1,\ldots,f_k)(x)\leq C\mathcal{M}(f_1,\ldots,f_k)(x)
\end{equation*}
for every $f_1,\ldots,f_k$ measurable on $X$ and for every $x\in X$.
\end{theorem}
\begin{proof}
Given $\varepsilon>0$, as before we may decompose the domain of any of the integrals defining $\Phi^*$ in dyadic annuli to obtain by the monotonicity of the $\varphi$, Lemma~\ref{lemma:JcomparableRadio} and Lemma~\ref{lemma:measureSectionsEonAhlfors}, the following basic estimate,
\begin{align*}
&\frac{1}{J(x,\varepsilon)}\int_{X^k}\varphi\left(\frac{\rho(x,x_1,\ldots,x_k)}{\varepsilon}\right)\prod_{i=1}^k|f_i(x_i)|d^k\mu(x_1,\ldots,x_k)\\
&=\frac{1}{J(x,\varepsilon)}\sum_{j\in \mathbb{Z}}\int\limits_{\{(x_1,\ldots,x_k):\varepsilon 2^{j}\leq \rho(x,x_1,\ldots,x_k)<\varepsilon 2^{j+1}\}}\varphi\left(\frac{\rho(x,x_1,\ldots,x_k)}{\varepsilon}\right)\prod_{i=1}|f_i(x_i)|d^k\mu(x_1,\ldots,x_k)\\
&\leq \frac{1}{C_1s(\varphi)}\sum_{j\in \mathbb{Z}}\varphi\left(2^{j}\right)\frac{1}{\varepsilon^{k\alpha}}\int\limits_{\{(x_1,\ldots,x_k):\varepsilon 2^{j}\leq \rho(x,x_1,\ldots,x_k)<\varepsilon 2^{j+1}\}}\prod_{i=1}^k|f_i(x_i)|d^k\mu(x_1,\ldots,x_k)\\
&\leq \frac{(2\kappa)^{k\alpha}\Gamma^k 2^{k\alpha}}{C_1s(\varphi)}\sum_{j\in \mathbb{Z}}\varphi\left(2^{j}\right)2^{k\alpha j}\frac{1}{\mu^k(E(x,\varepsilon2^{j+1}))}\int\limits_{\{\varepsilon 2^{j}\leq \rho(x,x_1,\ldots,x_k)<\varepsilon 2^{j+1}\}}\prod_{i=1}^k|f_i(x_i)|d^k\mu(x_1,\ldots,x_k)\\
&\leq \frac{(4\kappa)^{k\alpha}\Gamma^k}{C_1s(\varphi)}\left(\sum_{j\in \mathbb{Z}}\varphi\left(2^{j}\right)2^{k\alpha j}\right)\mathcal{M}(f_1,\ldots,f_k)(x).
\end{align*}
Let us now notice that 
\begin{align*}
	\sum_{j\in \mathbb{Z}}\varphi\left(2^{j}\right)2^{k\alpha j}&=\sum_{j\in \mathbb{Z}}\frac{1}{\log 2}\int_{2^{j-1}}^{2^j}\varphi(2^{j})2^{k\alpha j}\frac{dt}{t}\\
	&\leq \frac{2^{k\alpha}}{\log 2}\sum_{j\in \mathbb{Z}}\int_{2^{j-1}}^{2^j}\varphi(t)t^{k\alpha}\frac{dt}{t}\\
	&=\frac{2^{k\alpha}}{\log 2}\int_{0}^{\infty}\varphi(t)t^{k\alpha-1}dt\\
	&= \frac{2^{k\alpha}}{\log 2}s(\varphi),
\end{align*}
and the result is proved with constant $C=\frac{(8\kappa)^{k\alpha}\Gamma^k}{C_1\log 2}$.
\end{proof}	

From Lemma~\ref{lemma:sectionesentrebolas} we easily obtain the pointwise boundedness of $\mathcal{M}(f_1,\ldots,f_k)(x)$ in terms of the standard Hardy-Littlewood maximal operator on $(X,d,\mu)$ acting on each $f_i$, $i=1,\ldots,k$,
\begin{equation*}
	Mf_i(x)=\sup_{r>0}\frac{1}{\mu(B_d(x,r))}\int_{B_d(x,r)}\abs{f_i(y)}d\mu(y).
\end{equation*}
\begin{proposition}\label{propo:pointwiseEstimatebyMaximalHL}
	Let $(X,d,\mu)$ be a space of homogeneous type with doubling constant $A$. Then
	\begin{equation*}
		\mathcal{M}(f_1,\ldots,f_k)(x)\leq (2\kappa)^{k\log_2 A}\prod_{i=1}^k Mf_i(x).
	\end{equation*}
\end{proposition}
\begin{proof}
	From Lemma~\ref{lemma:sectionesentrebolas}, for each $x\in X$ and each $r>0$, we have
	\begin{align*}
		\frac{1}{\mu^k(E(x,r))}
		&\int_{E(x,r)}\prod_{i=1}^k |f_i(x_i)| d\mu^k(x_1,\ldots,x_k) \\ 
		&\leq
		\frac{1}{\prod_{i=1}^k\mu(B_d(x,r))}\int\limits_{\prod_{i=1}^k B_d(x,2\kappa r)}	\prod_{i=1}^k \abs{f_i(x_i)} d\mu(x_1)\ldots d\mu(x_k)\\
		&= \prod_{i=1}^k \frac{1}{\mu(B_d(x,r))}\int\limits_{B_d(x,2\kappa r)}\abs{f_i(x_i)} d\mu(x_i)\\
		&\leq (2\kappa)^{k\log_2 A}\prod_{i=1}^k Mf_i(x).
	\end{align*}
\end{proof}
Theorem~\ref{thm:PointwiseEstimateforPhistarbyMaximalHLsections}, the estimate obtained in Proposition~\ref{propo:pointwiseEstimatebyMaximalHL}, H\"{o}lder inequality and the boundedness of $M$ in the spaces $L^q(X,\mu)$ for $q>1$, give the basic boundedness properties of $\Phi^*$ on Lebesgue spaces.
\begin{theorem}\label{thm:BoundednessPhistarbyMaximalLebesgueSpaces}
Let $(X,d,\mu)$ be an $\alpha$-Ahlfors regular space and $\varphi:\mathbb{R}_{\geq 0}\to \mathbb{R}_{\geq 0}$ be a nonincreasing function such that $0<s(\varphi)=\int_0^\infty\varphi(t)t^{k\alpha-1}dt<\infty$. Let $1<p_j\leq \infty$, $j=1,\ldots,k$ and $1=\sum_{j=1}^k \frac{1}{p_j}$. Then there exists a constant $C$ depending on $\gamma$, $\Gamma$, $\alpha$, $k$ and the $p_j$'s such that 
\begin{equation*}
	\norm{\Phi^*(f_1,\ldots,f_k)}_{L^1(X,\mu)}\leq C \prod_{j=1}^k\norm{f_j}_{L^{p_j}(X,\mu)}
\end{equation*}
holds for every measurable functions $f_1,\ldots,f_k$ on $X$.
\end{theorem}

%%%%%%%%%%%%%%%%%%%%%%%%%%%%%%%%%%%%%%%%%%%%%%%%

\section{The approximation to the multilinear identity}\label{sec:TheApproximationtotheMultilinearIdentity}
In this section we shall prove the following result.
\begin{theorem}\label{thm:ApproximationIdentitySpaceLebesgue}
Let $(X,d,\mu)$ be a complete $\alpha$-Ahlfors regular space. Let $k$ be an integer larger than or equal to one. Let $\varphi:\mathbb{R}_{\geq 0}\to \mathbb{R}_{\geq 0}$ be a nonincreasing function such that $0<s(\varphi)=\int_0^\infty\varphi(t)t^{k\alpha-1}dt<\infty$. Let $1<p_j\leq \infty$,  $j=1,\ldots,k$, satisfying $1=\sum_{j=1}^k \frac{1}{p_j}$. Then 
\begin{align*}
	\lim_{\varepsilon\to 0}	\int_{X^k}\Phi_\varepsilon(x,x_1,\ldots,x_k)\left|\prod_{j=1}^kf_j(x_j)-\prod_{j=1}^k f_j(x)\right|d^k\mu(x_1,\ldots,x_k)=0,	
\end{align*}
for almost every $x\in X$ and every $f_j\in L^{p_j}(X,\mu)$, $j=1,\ldots,k$.
\end{theorem}
Notice that, in particular, the above result implies that \eqref{eq:multilinearoperatorphimeans} converges to $\prod_{j=1}^k f_j(x)$ when $\varepsilon$ tends to zero.

The main tool, as in the linear case, is the boundedness of $\Phi^*$ obtained in Theorem~\ref{thm:BoundednessPhistarbyMaximalLebesgueSpaces}. Let us start by stating and proving two auxiliary lemmas.
\begin{lemma}\label{lemma:DiferenciadeProductos}
	Let $\{a_i: i=1,\ldots,k\}$ and $\{b_i:i=1,\ldots,k\}$ be two finite sequences of real numbers with length $k$. Then
	\begin{equation*}
		\prod_{i=1}^k a_i - \prod_{i=1}^k b_i = \sum_{i=1}^k (a_i-b_i)\Bigl(\prod_{j=i+1}^k a_j\Bigr)\Bigl(\prod_{l=1}^{i-1}b_l\Bigr).
	\end{equation*}
\end{lemma}
\begin{proof}
	We proceed by induction on $k$. 	
	For $k=2$, we have $a_1a_2-b_1b_2 = (a_1-b_1)a_2 + b_1(a_2-b_2)$. Assume now that the formula holds true for $k$. Let us prove it for $k+1$,
	\begin{align*}
		\prod_{i=1}^{k+1}a_i - \prod_{i=1}^{k+1}b_i &= \Bigl(\prod_{i=1}^k a_i\Bigr) a_{k+1} - \Bigl(\prod_{i=1}^k b_i\Bigr) b_{k+1}\\
		&= \Bigl(\prod_{i=1}^{k}a_i - \prod_{i=1}^{k} b_i\Bigr)a_{k+1} + \prod_{i=1}^k b_i(a_{k+1}-b_{k+1})\\
		&= \sum_{i=1}^k(a_i-b_i)\Bigl(\prod_{l=1}^{i-1}b_l\Bigr)\Bigl(\prod_{j=i+1}^{k+1}a_j\Bigr) + \Bigl(\prod_{i=1}^kb_i\Bigr)(a_{k+1}-b_{k+1})\\
		&= \sum_{i=1}^{k+1}(a_i-b_i)\Bigl(\prod_{j=i+1}^{k+1}a_j\Bigr)\Bigl(\prod_{l=1}^{i-1}b_l\Bigr).
	\end{align*}
\end{proof}
\begin{lemma}\label{lemma:ApproximationIdentityContinuousCompactlySupported}
	Let $(X,d,\mu)$, $k$ and $\varphi$ as in Theorem~\ref{thm:ApproximationIdentitySpaceLebesgue}. Let $\psi_j$, $j=1,\ldots,k$ be $k$ real compactly supported and continuous functions defined on $X$. Then
	\begin{align*}
		\lim_{\varepsilon\to 0}	\int_{X^k}\Phi_\varepsilon(x,x_1,\ldots,x_k)\left|\prod_{j=1}^k\psi_j(x_j)-\prod_{j=1}^k \psi_j(x)\right|d^k\mu(x_1,\ldots,x_k)=0,
	\end{align*}
	for every $x\in X$.
\end{lemma}
\begin{proof}
Notice first that the function $\prod_{j=1}^k\psi_j(x_j)$ defined on $X^k$ is continuous and compactly supported. Then given $\eta>0$, there exists $\sigma>0$ such that $|\prod_{j=1}^k\psi_j(x_j)-\prod_{j=1}^k\psi_j(x)|<\eta$ whenever $(x_1,\ldots,x_k)\in \prod_{i=1}^kB_d(x,\sigma)$ or equivalently $x_j\in B_d(x,\sigma)$. With this value of $\sigma$, let us split the integral in the following way 
	\begin{align*}
		\int_{X^k}&\Phi_\varepsilon(x,x_1,\ldots,x_k)\left|\prod_{j=1}^k\psi_j(x_j)-\prod_{j=1}^k\psi_j(x)\right|d^k\mu(x_1,\ldots,x_k) \\
		&\leq\frac{1}{J(x,\varepsilon)}\int_{\prod_{i=1}^kB_d(x,\sigma)}\varphi\left(\frac{\rho(x,x_1,\ldots,x_k)}{\varepsilon}\right)\left|\prod_{j=1}^k\psi_j(x_j)-\prod_{j=1}^k\psi_j(x)\right|d^k\mu(x_1,\ldots,x_k)\\
		&\phantom{\leq\frac{1}{J(x,\varepsilon)}} + \frac{2\prod_{j=1}^k\norm{\psi_j}_{L^\infty(X,\mu)}}{J(x,\varepsilon)}\int\limits_{X^k\setminus \prod_{i=1}^kB_d(x,\sigma)}\varphi\left(\frac{\rho(x,x_1,\ldots,x_k)}{\varepsilon}\right)d^k\mu(x_1,\ldots,x_k).
	\end{align*}
The first term above is bounded by $\eta$. In order to show that the second term is bounded by $\eta$ taking $\varepsilon$ small enough, notice first that from Lemma~\ref{lemma:sectionesentrebolas} we have that $X^k\setminus \prod_{i=1}^kB_d(x,\sigma)$ is contained in $X^k\setminus E(x,\tfrac{\sigma}{2\kappa})$. Hence, from Lemmas~\ref{lemma:measureSectionsEonAhlfors} and \ref{lemma:JcomparableRadio}, 
\begin{align*}
	&\frac{2\prod_{j=1}^k\norm{\psi_j}_{L^\infty(X,\mu)}}{J(x,\varepsilon)}\int\limits_{X^k\setminus \prod_{i=1}^kB_d(x,\sigma)}\varphi\left(\frac{\rho(x,x_1,\ldots,x_k)}{\varepsilon}\right)d^k\mu(x_1,\ldots,x_k)\\
	&\leq \frac{2\prod_{j=1}^k\norm{\psi_j}_{L^\infty(X,\mu)}}{J(x,\varepsilon)}\int\limits_{\rho(x,x_1,\ldots,x_k)\geq\tfrac{\sigma}{2\kappa}}\varphi\left(\frac{\rho(x,x_1,\ldots,x_k)}{\varepsilon}\right)d^k\mu(x_1,\ldots,x_k)\\
	&\leq \frac{2\prod_{j=1}^k\norm{\psi_j}_{L^\infty(X,\mu)}}{J(x,\varepsilon)}\!\!\sum_{i\geq\log_2\tfrac{\sigma}{2\kappa\varepsilon}-1}\int\limits_{2^i\leq\frac{\rho(x,x_1,\ldots,x_k)}{\varepsilon}<2^{i+1}}\!\!\!\varphi\left(\frac{\rho(x,x_1,\ldots,x_k)}{\varepsilon}\right)d^k\mu(x_1,\ldots,x_k)\\
	&\leq 2\prod_{j=1}^k\norm{\psi_j}_{L^\infty(X,\mu)}\sum_{i\geq\log_2\tfrac{\sigma}{2\kappa\varepsilon}-1} \varphi(2^i)\frac{\mu^k(E(x,\varepsilon 2^{i+1}))}{J(x,\varepsilon)}\\
	&\leq 2\prod_{j=1}^k\norm{\psi_j}_{L^\infty(X,\mu)} \frac{2^{k\alpha}(2\kappa)^{k\alpha}\Gamma^k}{C_1s(\varphi)}\sum_{i\geq\log_2\tfrac{\sigma}{2\kappa\varepsilon}-1} \varphi(2^i) 2^{ik\alpha}\\
	&\leq 2\prod_{j=1}^k\norm{\psi_j}_{L^\infty(X,\mu)} \frac{2^{k\alpha}(2\kappa)^{k\alpha}\Gamma^k}{C_1s(\varphi)}\frac{1}{\log 2}\int_{\log_2\tfrac{\sigma}{2\kappa\varepsilon}-1}^\infty \varphi(t) t^{k\alpha-1} dt,
	\end{align*}
 which tends to zero for $\varepsilon\to 0$, since $s(\varphi)=\int_0^\infty\varphi(t)t^{k\alpha-1}dt$ is finite.
\end{proof}
We are now in position to prove Theorem~\ref{thm:ApproximationIdentitySpaceLebesgue}.

\begin{proof}[Proof of Theorem~\ref{thm:ApproximationIdentitySpaceLebesgue}]
Let us start noticing that it suffices to show that for $f_j\in L^{p_j}(X,\mu)$ and every $m\geq 1$,
\begin{align*}
	\mu\biggl(\bigg \{x \in X: \limsup_{\varepsilon \to 0}\int_{X^k}\Phi_\varepsilon(x,x_1,\ldots,x_k)\left|\prod_{j=1}^kf_j(x_j)\!-\!\prod_{j=1}^k f_j(x)\right|d\mu^k(x_1,\ldots,x_k)>\frac{1}{m}\bigg \}\biggr)=0.
\end{align*}
Now, from Lemma~\ref{lemma:ApproximationIdentityContinuousCompactlySupported}, with $\psi_1,\ldots,\psi_k$ compactly supported continuous functions, we have that the above set is contained in 
\begin{align*}
	&\bigg \{x \in X: \limsup_{\varepsilon \to 0}\int_{X^k}\Phi_\varepsilon(x,x_1,\ldots,x_k)\left|\prod_{j=1}^kf_j(x_j)-\prod_{j=1}^k\psi_j(x_j)\right|d^k\mu(x_1,\ldots,x_k)\\
	&\phantom{\bigg \{x \in \quad} +\limsup_{\varepsilon \to 0}\int_{X^k}\Phi_\varepsilon(x,x_1,\ldots,x_k)\left|\prod_{j=1}^k\psi_j(x_j)-\prod_{j=1}^k\psi_j(x)\right|d^k\mu(x_1,\ldots,x_k)\\
	&\phantom{\bigg \{x \in \quad} +\left|\prod_{j=1}^k\psi_j(x)-\prod_{j=1}^k f_j(x)\right|>\frac{1}{m}\bigg \}\\
	&\subset \bigg \{x \in X: \sup_{\varepsilon >0}\int_{X^k}\Phi_\varepsilon(x,x_1,\ldots,x_k)\left|\prod_{j=1}^kf_j(x_j)-\prod_{j=1}^k\psi_j(x_j)\right|d^k\mu(x_1,\ldots,x_k)>\frac{1}{2m}\bigg \}\\
	&\phantom{\bigg \{x \in \quad}\cup \bigg \{x \in X: \left|\prod_{j=1}^k\psi_j(x)-\prod_{j=1}^k f_j(x)\right|>\frac{1}{2m}\bigg \}\\
	&=:A\cup B.
\end{align*}	
Let us first show that the measure of $B$ can be made as small as desired by an appropriate choice of $\psi_1,\ldots,\psi_k$. From Lemma~\ref{lemma:DiferenciadeProductos} we have
\begin{align*}
\mu(B) &\leq 2m \int_X\left|\prod_{j=1}^kf_j(x)-\prod_{j=1}^k\psi_j(x)\right|d\mu(x)\\
&\leq 2m\sum_{j=1}^k\int_X\left|(f_j(x)-\psi_j(x))\left(\prod_{i=j+1}^kf_i(x)\right)\left(\prod_{l=1}^{j-1}\psi_l(x)\right)\right|d\mu(x).
\end{align*}
Now, from H\"{o}lder's inequality,
\begin{equation*}
\mu(B)\leq 2m\sum_{j=1}^k\norm{f_j-\psi_j}_{L^{p_j}(X,\mu)}\left(\prod_{i=j+1}^k\norm{f_i}_{L^{p_i}(X,\mu)}\right)\left(\prod_{l=1}^{j-1}\norm{\psi_l}_{L^{p_l}(X,\mu)}\right),
\end{equation*}
which is small if we choose $\psi_j$ as close as needed to $f_j$ in the $L^{p_j}(X,\mu)$ norm. For the bound of $A$ we shall use Theorem~\ref{thm:BoundednessPhistarbyMaximalLebesgueSpaces}. In fact, applying again Lemma~\ref{lemma:DiferenciadeProductos}, we have
\begin{align*}
&A = \bigg \{x:\limsup_{\varepsilon \to 0}\int\limits_{X^k}\Phi_{\varepsilon}(x,x_1,\ldots,x_k)\bigg|\sum_{j=1}^k (f_j(x_j)-\psi_j(x_j))\prod_{i=j+1}^k f_i(x_j)\prod_{l=1}^{j-1}\psi_l(x_j)\bigg|d^k\mu>\frac{1}{2m}\bigg \}\\
&\subset \bigg \{x: \limsup_{\varepsilon \to 0}\sum_{j=1}^k \int\limits_{X^k}\Phi_{\varepsilon}(x,x_1,\ldots,x_k) |f_j(x_j)-\psi_j(x_j)|\prod_{i=j+1}^k |f_i(x_j)|\prod_{l=1}^{j-1}|\psi_l(x_j)|d^k\mu>\frac{1}{2m}\bigg \}\\
&\subset \bigg \{x: \sum_{j=1}^k \Phi^*(\psi_1,\ldots,\psi_{j-1},f_j-\psi_j,f_{j+1},\ldots,f_k)(x)>\frac{1}{2m}\bigg \}\\
&\subset \bigcup_{j=1}^k\bigg \{x: \Phi^*(\psi_1,\ldots,\psi_{j-1},f_j-\psi_j,f_{j+1},\ldots,f_k)(x)>\frac{1}{2km}\bigg \}.
\end{align*}
So that, from Chebychev inequality and Theorem~\ref{thm:BoundednessPhistarbyMaximalLebesgueSpaces},
\begin{align*}
	\mu(A) &\leq 2km \sum_{j=1}^k \|\Phi^*(\psi_1,\ldots,\psi_{j-1},f_j-\psi_j,f_{j+1},\ldots,f_k)\|_{L^1(X,\mu)}\\	
	&\leq 2kmC\sum_{j=1}^k \|\psi_1\|_{L^{p_1}} \ldots \|\psi_{j-1}\|_{L^{p_{j-1}}} \|f_j-\psi_j\|_{L^{p_{j}}} \|f_{j+1}\|_{L^{p_{j+1}}}\ldots \|f_k\|_{L^{p_k}}.
\end{align*}
Choosing $\psi_j$ with $\|f_j-\psi_j\|_{L^{p_j}}$ small enough for $j=1,\ldots,k$ we get the desired result. 
\end{proof}

%%%%%%%%%%%%%%%%%%%%%%%  References %%%%%%%%%%%%%%%%%

%\bibliographystyle{amsalpha}
%\bibliography{ref}

\providecommand{\bysame}{\leavevmode\hbox to3em{\hrulefill}\thinspace}
\providecommand{\MR}{\relax\ifhmode\unskip\space\fi MR }
% \MRhref is called by the amsart/book/proc definition of \MR.
\providecommand{\MRhref}[2]{%
	\href{http://www.ams.org/mathscinet-getitem?mr=#1}{#2}
}
\providecommand{\href}[2]{#2}

%%%%%%%%%%%%%%%%%%%%% Declarations %%%%%%%%%%%%%%%%%%%%%%%%%%%%%%%%%%%%%%

%\section*{Declarations}
%
\subsection*{Acknowledgements}
This work was supported by Consejo Nacional de Investigaciones Cient\'ificas y T\'ecnicas-CONICET in Argentina, 
Grant PIP-2021-2023-11220200101940CO.

%%%%%%%%%%%%% Affiliation and Address Authors %%%%%%%%%%%%%%%

\bigskip

\medskip

\noindent{\footnotesize
	\noindent\textit{Affiliation.\,}
	\textsc{Instituto de Matem\'{a}tica Aplicada del Litoral ``Dra. Eleonor Harboure'', UNL, CONICET.}

	\smallskip
	\noindent\textit{Address.\,}\textmd{IMAL, Streets F.~Leloir and A.P.~Calder\'on, CCT CONICET Santa Fe, Predio ``Alberto Ca\-ssa\-no'', Colectora Ruta Nac.~168 km~0, Paraje El Pozo, S3007ABA Santa Fe, Argentina.}

	\smallskip
	\noindent \textit{E-mail address.\, }\verb|haimar@santafe-conicet.gov.ar|; \\
	\hspace*{2.5cm} \verb|ivanagomez@santafe-conicet.gov.ar|;\\ 
	\hspace*{2.5cm} \verb|joaquintoledo@santafe-conicet.gov.ar| 
}

\end{document}